\documentclass[review,3p,10pt]{elsarticle}
\biboptions{sort&compress}

\usepackage{graphicx}
\usepackage{epstopdf}
\usepackage{amsmath}
\usepackage{cases}
\usepackage[caption=false,font=footnotesize]{subfig}
\usepackage{fixltx2e}
\usepackage{amssymb}
\usepackage{mathtools}
\usepackage{bm}
\usepackage{multirow}
\usepackage{color}
\usepackage{xcolor}
\usepackage{algorithm}
\usepackage{algorithmic}
\usepackage{amsthm}
\usepackage{mathrsfs}
\usepackage{textcomp,mathcomp}
\usepackage{threeparttable}
\usepackage{bbding}
\usepackage{pifont}
\usepackage{multicol}
\usepackage{enumitem}
\usepackage{adjustbox}

\newtheorem{assumption}{Assumption}
\newtheorem{proposition}{Proposition}
\newtheorem{remark}{Remark}

\newcommand{\tabincell}[2]{\renewcommand\arraystretch{0.9}\begin{tabular}{@{}#1@{}}#2\end{tabular}}

\journal{Energy}

\begin{document}
\begin{frontmatter}

\title{Carbon-Driven Incentive Mechanism for Renewable Power-to-Ammonia Production in Coupled Carbon and Ammonia Markets}

\author[label1]{Yangjun~Zeng}
\author[label2]{Huayan~Geng\corref{cor1}}
\ead{genghuayan@swufe.edu.cn}
\author[label1]{Yiwei~Qiu\corref{cor2}}
\ead{ywqiu@scu.edu.cn}
\author[label2]{Xiuli~Sun}
\author[label1]{Liuchao~Xu}
\author[label3]{Jiarong~Li}
\author[label1]{Shi~Chen}
\author[label1]{Buxiang~Zhou}
\author[label1]{Kaigui~Xie}

\address[label1]{College of Electrical Engineering, Sichuan University, Chengdu, 610065, China}
\address[label2]{School of Statistics and Data Science, Southwestern University of Finance and Economics, Chengdu, 611130, China}
\address[label3]{Harvard John A. Paulson School of Engineering and Applied Sciences, Harvard University, Cambridge 02138, USA}
\cortext[cor1]{Corresponding author}
\cortext[cor2]{Corresponding author}

\begin{abstract}
  Renewable power-to-ammonia (ReP2A) production offers a promising pathway to decarbonize the power, transport and, chemical sectors, yet its competitiveness remains limited by high costs and fragmented carbon-policy frameworks. In particular, a unified mechanism that links ReP2A producers with fossil-based gray ammonia (GA) competitors in carbon and ammonia markets, while coordinating incentives among renewable generation, hydrogen production, and ammonia synthesis stakeholders in the ReP2A process chain, is still lacking.
  To address this gap, this paper proposes a hierarchical carbon-driven incentive mechanism (PCIM) that integrates carbon policy with multi-energy market interactions. A two-layer trading framework is developed, where ReP2A and GA compete in carbon allowance (CA) and ammonia markets (outer layer), while electricity and hydrogen transactions coordinate the ReP2A chain (inner layer). The resulting interactions are modeled as a hierarchical equilibrium, where the inner layer is reformulated as a tractable equivalent optimization problem, and the outer layer is solved as a mixed-integer linear program (MILP) derived from Karush-Kuhn-Tucker conditions.
  Based on equilibrium analysis, the carbon-related revenue of ReP2A is quantified, and a CA allocation mechanism (PCAM) is proposed to ensure individually rationality among stakeholders. Results show that the proposed mechanism reduces carbon emissions by 12.9\% with only a 1.8\% decrease in sector-wide revenue. Moreover, carbon pricing under the proposed framework redistributes profits between green and gray ammonia without reducing total welfare, and the PCAM further enhances stakeholders' willingness to participate in ReP2A production.
\end{abstract}

\end{frontmatter}

\break
\section*{Nomenclature}

\addcontentsline{toc}{section}{Nomenclature}
\begin{multicols}{2}
	
	\footnotesize
	\setlength{\columnsep}{18pt}

\noindent\textbf{Abbreviations}
\begin{description}[leftmargin=2.2cm, labelwidth=2.0cm, labelsep=0.2cm, align=left, itemsep=0pt, topsep=0pt, font=\normalfont]
	\item[AE, ae] Alkaline electrolyzer
	\item[AST, ast] Ammonia storage
	\item[ASY, asy] Ammonia synthesis
	\item[BES, bes] Battery energy storage
	\item[CA] Carbon allowance
	\item[GA, ga] Gray ammonia stakeholder
	\item[HP, hp] Hydrogen production stakeholder
	\item[HST, hst] Hydrogen storage
	\item[PV, pv] Photovoltaic
	\item[RA, ra] Renewable ammonia stakeholder
	\item[RG, rg] Renewable power generation stakeholder
	\item[WT, wt] Wind turbine
\end{description}
\vspace{5pt}

\noindent\textbf{Indices}

\begin{description}[leftmargin=2.3cm, labelwidth=2.1cm, labelsep=0.2cm, align=left, itemsep=0pt, topsep=0pt, font=\normalfont]
	\addcontentsline{toc}{section}{Nomenclature}
	
	\item[$w,t$] Index for weeks and time intervals
	\item[$i, j, j'$] Index for buses
	\item[$ij$]   Index for the branch between buses $i$ and $j$
	\item[$m,n$]   Index for hydrogen nodes
	\item[$mn$]   Index for the pipelines between nodes $m$ and $n$
\end{description}

\noindent\textbf{Variables}

\noindent\textbf{1) Price-related variables}

\begin{description}[leftmargin=2.3cm, labelwidth=2.1cm, labelsep=0.2cm, align=left, itemsep=0pt, topsep=0pt, font=\normalfont]
	\addcontentsline{toc}{section}{Nomenclature}
	\item[$\rho_{t}^{\text{rg-hp/ra,e}}$] Electricity prices between RG and HP/RA
	\item[$\rho_{t}^{\text{hp-as,h}}$] Hydrogen price between HP and AS
	\item[$\rho_{}^{\text{CA}}$] CA price 
	\item[$\rho_{w}^{\text{am}}$] Ammonia price
\end{description}

\noindent\textbf{2) Variables related to RG}

\begin{description}[leftmargin=2.3cm, labelwidth=2.1cm, labelsep=0.2cm, align=left, itemsep=0pt, topsep=0pt, font=\normalfont]
	\addcontentsline{toc}{section}{Nomenclature}
	\item[$P_{t}^{\text{rg,wt}}, Q_t^{\text{rg,wt}}$]  Active and reactive power of WT \vspace{1pt}
	\item[$P_{t}^{\text{rg,pv}}, Q_t^{\text{rg,pv}}$]  Active and reactive power of PV \vspace{1pt}
	\item[$P_{t}^{\text{rg,wt/pv,curt}}$] Power curtailment of WT/PV
	\item[$P_{t}^{\text{rg,bes,c/d}}$] BES charging/discharging power in RG \vspace{1pt}
	\item[$P_{t}^{\text{rg,sell,hp/ra}}$] Power that RG sells to HP/RA \vspace{1pt}
	\item[$S^{\text{rg,bes}}_t,Q_{t}^{\text{rg,bes}}$] State and reactive power of BES in RG \vspace{1pt}
	\item[$P_{ij,t}, Q_{ij,t}$] Active/reactive power flows on branch $ij$
	\item[$\upsilon_{j,t}$] Square of the voltage amplitude at bus $j$
	\item[$q^{\text{rg}}$]  CA sold from RG to a GA
\end{description}

\noindent\textbf{3) Variables related to HP}

\begin{description}[style=standard, leftmargin=2.3cm, labelwidth=2.1cm, labelsep=0.2cm, align=left, itemsep=0pt, topsep=0pt, font=\normalfont]
	\addcontentsline{toc}{section}{Nomenclature}
	\item[$P_{t}^{\text{hp,buy,rg}}$] Power bought by HP from RG \vspace{1pt}
	\item[$P_{t}^{\text{hp,ae/comp}}$] Power of AE and hydrogen compressor \vspace{1pt}
	\item[$P_{t}^{\text{hp,bes,c/d}}$] BES charging/discharging power in HP \vspace{1pt}
	\item[$S^{\text{hp,bes}}_t,Q_{t}^{\text{hp,bes}}$]  State and reactive power of BES in HP \vspace{1pt}
	\item[$f^{\text{hp,pro}}_t$]  Hydrogen production rate\vspace{1pt}
	\item[$f^{\text{hp,sell,ra}}_t$]  Hydrogen sold from HP to RA\vspace{1pt}
	\item[$f^{\text{hp,hst,in/out}}_{t}$]  Hydrogen inflow/outflow of HST in HP
	\item[$F_{mn,t}$]  Average hydrogen flow of pipeline $mn$
	\item[$F_{mn,t}^{\text{in/out}}$]  Hydrogen inflow/outflow of pipeline $mn$
	\item[$p_{m,t}$]  Pressure at hydrogen node $m$
	\item[$LP_{mn,t}$]  Linepack storage of pipeline $mn$
	\item[$q^{\text{hp}}$] CA sold from HP to GA
\end{description}

\noindent\textbf{4) Variables related to RA}

\begin{description}[leftmargin=2.3cm, labelwidth=2.1cm, labelsep=0.2cm, align=left, itemsep=0pt, topsep=0pt, font=\normalfont]
	\addcontentsline{toc}{section}{Nomenclature}
	\item[$P_{t}^{\text{ra,asy}}$] Power consumption of ASY in RA
	\item[$P_{t}^{\text{ra,buy,rg}}$] Power bought by RA from RG
	\item[$P_{t}^{\text{ra,back}}$] Backup power for continuous operation of ASY
	\item[$f^{\text{ra,buy,hp}}_t$]  Hydrogen bought by AS from HP \vspace{1pt}
	\item[$f^{\text{ra,hst,in/out}}_{t}$]  Hydrogen inflow/outflow of HST in RA\vspace{1pt}
	\item[$f^{\text{ra,use}}_t$]  Hydrogen consumption for ASY
	\item[$S^{\text{ra,hst}}_t,S^{\text{ra,ast}}_w$]  States of HST and AST in RA \vspace{1pt}
	\item[$M^{\text{ra,pro}}_{t}$]  Ammonia flow rate of RA
	\item[$D^{\text{ra,sell}}_{w}$]  Ammonia sales volume of RA
	\item[$q^{\text{ra}}$]  CA sold from RA to GA
\end{description}

\noindent\textbf{5) Variables related to GA}

\begin{description}[leftmargin=2.3cm, labelwidth=2.1cm, labelsep=0.2cm, align=left, itemsep=0pt, topsep=0pt, font=\normalfont]
	\addcontentsline{toc}{section}{Nomenclature}
	\item[$M^{\text{ga,pro}}_{t}$] GA production rate
	\item[$D^{\text{ga,sell}}_{w}$] GA sales volume
	\item[$q^{\text{ga}}$]  CA purchased by GA from the ReP2A system
	\item[$q^{\text{emis}}_t$] Carbon emissions from GA
\end{description}
\vspace{5pt}

\noindent\textbf{Parameters}

\begin{description}[leftmargin=2.3cm, labelwidth=2.1cm, labelsep=0.2cm, align=left, itemsep=0pt, topsep=0pt, font=\normalfont]
	\addcontentsline{toc}{section}{Nomenclature}
	\item[$T, \tau, \Delta t$] Operational horizon, subhorizon and step length
	\item[$W^{\text{rg,wt/pv/bes}}$] WT/PV/BES installed capacities in RG
	\item[$W^{\text{hp,bes/ae/hst}}$] BES/AE/HST installed capacities in HP
	\item[$W^{\text{ra,hst/asy/ast}}$] HST/ASY/AST installed capacities in RA
	\item[$W^{\text{ga,asy}}$] ASY installed capacities in GA
	\item[$\rho^{\text{max}}, k^{\text{am}}$] Parameters in demand--price relationship
	\item[$c^{\text{ga}}, k^{\text{emis}}$] GA production cost and carbon-emission factor
	\item[$q^{\text{allo}},q^{\text{rewa}}$] Initial CA for the GA/ReP2A system
	\item[$P_{t}^{\text{rg,wt/pv,max}}$] Maximum power of WT/PV
	\item[$\eta^{\text{bes,c/d}}$] BES charging/discharging efficiencies
	\item[$\overline{\eta}^{\text{bes}},\underline{\eta}^{\text{bes}}$] State limits of BES
	\item[$\zeta^{\text{bes}},\sigma^{\text{deg}}$] BES self-discharge ratio and degradation cost
	\item[$\overline{\upsilon}_{j},\underline{\upsilon}_{j}$] Voltage magnitude limits at bus $j$
	\item[$\overline{\eta}^{\text{ae}}, \underline{\eta}^{\text{ae}}$] Power limits of hydrogen production plant
	\item[$\eta^{\text{p2h}}$] Energy conversion coefficient of AE
	\item[$\eta^{\text{comp}}$] Compressor power consumption coefficient
	\item[$\eta^{\text{h2a}},\eta^{\text{p2a}}$] ASY hydrogen/power consumption coefficient
	\item[$\overline{\eta}^{\text{hst}},\underline{\eta}^{\text{hst}}$] State limits of HST
	\item[$\overline{p}_{m},\underline{p}_{m}$] Limits of squared pressure at hydrogen node $m$
	\item[$\gamma$] Penalty factor for hydrogen pressure deviation
	\item[$K_{mn}^{\text{lp}},K_{mn}^{\text{gf}}$] Weymouth constants of pipeline $mn$
	\item[$\overline{\eta}^{\text{asy}},\underline{\eta}^{\text{asy}}$] Production rate limits of ASY
	\item[$\overline{r}^{\text{asy}},\underline{r}^{\text{asy}}$] Maximum ramping limits of ASY
	\item[$\rho_t^{\text{as,back}}$] Backup power cost in RA
\end{description}

\end{multicols}

\break
\section{Introduction}
\label{sec:intro}

\subsection{Background and motivation}
\label{sec:background}

Amid accelerating climate change, reducing carbon emissions has become a priority across sectors. Carbon-related policies \cite{2024china, chyong2025energy, 2025greenammonia, lyu2025cost} now support the transition of fossil fuel-based industries toward renewable energy. Renewable power-to-ammonia (ReP2A), which uses renewable electricity to produce hydrogen for ammonia synthesis, provides a promising pathway to decarbonize the power, chemical, and shipping sectors \cite{li2025redesigning, zeng2025planning, wu2025dispatchable, zhao2024role} and has attracted increasing interest.

Large-scale ReP2A projects are emerging worldwide, including in Saudi Arabia \cite{2025NEOM}, Denmark \cite{campion2023techno}, and China's Inner Mongolia \cite{2025siziwang} and Jilin \cite{zeng2025optimal}. Despite this progress, high capital and operating costs limit the competitiveness of green ammonia compared to fossil-based gray ammonia (GA) \cite{olabi2023recent, lee2024comparative}. For example, in regions of China with abundant renewable resources, green ammonia costs can reach about 3,600~CNY/t \cite{yu2024optimal}, compared with approximately 2,000~CNY/t for GA \cite{shin2025comparative}. To close this gap, authorities employ carbon taxes \cite{olsen2018optimal}, subsidies \cite{chyong2025energy}, and carbon allowance (CA) schemes \cite{2024beijing}.

The power sector typically applies a cap-and-trade mechanism \cite{lin2025energy, Chen2021Conjectural} to regulate total emissions by allocating tradeable CAs to producers. Producers receive initial allowances and can trade them to optimize production. Extending this mechanism to the ammonia sector, where carbon trading is still limited, could enhance green ammonia competitiveness while constraining emissions from GA. Due to high transportation costs, ammonia markets are regional and typically oligopolistic \cite{ellwanger2023cost}. In such settings, competition between green and gray ammonia, combined with carbon trading, directly affects production decisions, prices, and emissions.

Beyond the competition between green and gray ammonia,
ReP2A itself is a multi-stage chain consisting of renewable power generation (RG),
hydrogen production (HP), and renewable ammonia (RA) synthesis. These stakeholders engage in electricity and hydrogen transactions and may face conflicting interests \cite{zeng2025planning, yu2023optimal}. The overall interactions are shown in Fig.~\ref{fig:system}. The following three key questions arise: 
\begin{itemize}
	\item How do the carbon and ammonia markets interact?
	
	\item How are internal electricity and hydrogen transactions in the ReP2H chain linked to the external carbon and ammonia markets?
	
	\item More interestingly, because carbon-reduction benefits result jointly from RG, HP, and RA, how should carbon revenue be allocated to ensure individual rationality (IR)?
\end{itemize}

\subsection{Literature review}
\label{sec:review}

Growing interest in green ammonia has led to the extensive study of its decarbonization potential and economic viability.
Olabi \textit{et al.} \cite{olabi2023recent} reviewed recent progress and tradeoffs among production routes, identifying low efficiency and high cost as primary barriers.
Del~Pozo \textit{et al.} \cite{del2022techno} assessed green ammonia as an energy carrier and reported that low-cost renewable ammonia supports interregional trade-driven decarbonization.
Chyong \textit{et al.} \cite{chyong2025energy} evaluated low-carbon ammonia under subsidies and carbon pricing, showing that the flexibility of the Haber-Bosch process is the key factor.
Egerer \textit{et al.} \cite{egerer2023economics} developed a trade model and reported that high carbon prices are crucial for the competitiveness of green ammonia.
Overall, the economic viability of green ammonia depends on its carbon advantage.

\begin{table}[t]
	\centering
	\scriptsize
	\caption{Summary and comparison of energy transactions in related literature}\vspace{6pt}
	\label{tab:literature}
	\renewcommand{\arraystretch}{1.35} 
	\begin{tabular}{c@{\hspace{5pt}}c@{\hspace{5pt}}c@{\hspace{5pt}}c@{\hspace{5pt}}c @{\hspace{5pt}}c @{\hspace{5pt}}c@{\hspace{5pt}}c}
		\hline\hline
		\multirow{2}{*}{Literature} & \multicolumn{4}{c}{Transactions} & \multirow{2}{*}{Settlement} & \multirow{2}{*}{Operational horizon} & \multirow{2}{*}{Game model} \\
		\cline{2-5}
		& Electricty & Hydrogen & Ammonia & Carbon & & & \\
		\hline
		Zeng et al. \cite{zeng2025planning} & \checkmark & \checkmark & $\times$ & $\times$ & Hourly &  \tabincell{c}{Annual \\(12 typical weeks)} & Nash game \\
		
		Chen et al. \cite{Chen2021Conjectural} & \checkmark & $\times$ & $\times$ & \checkmark & \tabincell{c}{Hourly (elec.),\\ Annual (carbon)} & \tabincell{c}{Annual\\(typical scenarios)} & \tabincell{c}{Conjectural-variations\\equilibrium} \\
		
		Yu et al. \cite{yu2023optimal} & \checkmark & \checkmark & $\times$ & $\times$ & Hourly   &  \tabincell{c}{Annual\\(8760 hours)}   & \tabincell{c}{Cooperative\\operation}\\
		
		Wang et al. \cite{wang2025collaborative} & $\times$ & \checkmark & $\times$ & \checkmark & Hourly & Daily  & \tabincell{c}{Cooperative game,\\Stackelberg game} \\
		
		Shi et al. \cite{shi2026value} & \checkmark & \checkmark & $\times$ & \checkmark & - & Dynamic evolution & Evolutionary game \\
		
		Tostado et al. \cite{tostado2024local} & \checkmark & \checkmark & $\times$ & $\times$ & Hourly  & Daily  & Stackelberg game \\
		
		Wang et al.\cite{wang2025electricity} & \checkmark & $\times$ & $\times$ & \checkmark    & Hourly  & Daily  & \tabincell{c}{Coalitional game} \\
		
		Xiang et al. \cite{xiang2024carbon} & \checkmark & $\times$ & $\times$ & \checkmark & Hourly   & Daily  & \tabincell{c}{Nash bargaining} \\
		
		Zhou et al. \cite{zhou2025novel} & \checkmark & $\times$ & $\times$ & \checkmark & Hourly  & Daily  & Stackelberg game \\
		
		Zhou et al. \cite{zhou2024joint} & \checkmark & $\times$ & $\times$ & \checkmark & Hourly  & Daily  & Noncooperative game \\
		
		Mu et al. \cite{mu2023decentralized} & \checkmark & $\times$ & $\times$ & \checkmark   & Hourly  & Daily  & Cooperative game \\
		\hline
		\textbf{\tabincell{c}{This paper}} & \checkmark & \checkmark & \checkmark & \checkmark & \tabincell{c}{Hourly (elec./hyd.),\\ Weekly (amm.),\\Annual (carbon)} &\tabincell{c}{Annual\\ (12 typical weeks)} & \tabincell{c}{Nash game,\\Cournot game,\\cooperative allocation} \\
		\hline\hline
	\end{tabular}
\end{table}

Building upon this, hydrogen transaction, as a crucial intermediate link within the ReP2A chain, has been extensively investigated across transportation \cite{wang2025collaborative}, chemical \cite{shi2026value}, industrial \cite{tostado2024local}, and building \cite{wang2025electricity} sectors.
For example, Wang \textit{et al.} \cite{wang2025collaborative} investigated the cooperation and CA allocation mechanisms among gray, blue, and green hydrogen, as well as their multi-level interactions with hydrogen refueling stations and hydrogen fuel cell vehicles.
Shi \textit{et al.} \cite{shi2026value} studied the  strategic evolution of the renewable power-to-methanol value chain based on an evolutionary game, considering the coupling of electricity, hydrogen, and green certificate markets.
Tostado-V{\'e}liz \textit{et al.} \cite{tostado2024local} proposed a local electricity-hydrogen trading framework for industrial parks, achieving coordination and benefit improvement among subsystems through a Stackelberg game.
Wang \textit{et al.} \cite{wang2025electricity} developed a coalition game-based electricity-hydrogen-heat coupled sharing framework for multi-energy buildings, promoting energy sharing and reducing carbon emissions.

Meanwhile, carbon markets are becoming more mature and influential in energy systems.
Xiang \textit{et al.} \cite{xiang2024carbon} investigated nodal carbon pricing for coordinating prosumers through Nash bargaining,
and Zhou \textit{et al.} \cite{zhou2025novel} developed a four-layer
distributed game framework for microgrids.
To study interactions between carbon and electricity-gas markets, Chen \textit{et al.} \cite{Chen2021Conjectural} introduced a conjectural-variation equilibrium model, whereas Zhou \textit{et al.} \cite{zhou2024joint} applied reinforcement learning to clear the joint energy-carbon market.
Mu \textit{et al.} \cite{mu2023decentralized} proposed a decentralized electricity-CA trading framework. These studies focus mainly on power systems; in contrast, the ReP2A chain combines power, hydrogen, ammonia, and carbon, with stakeholders exhibiting heterogeneous flexibility, which remains unexplored. A comparative summary of the related literature is provided in Table \ref{tab:literature}.

Fair CA allocation and transactions are also critical for decarbonizing the chemical industry \cite{2024carbon}. Existing CA allocation methods include grandfathering and benchmarking \cite{zhang2024grandfather}. However, gray ammonia is not yet included in carbon markets, and emerging ReP2A projects complicate carbon management. Some policies \cite{2025greenammonia} now grant carbon credits for green ammonia and allow the credits to be traded in carbon markets, creating the need for coordinated carbon management across both industries.

Internal interactions among the stakeholders along the ReP2A chain have also been studied.
Yu \textit{et al.} \cite{yu2023optimal} proposed a sizing and pricing model ensuring collective and individual benefits among RG, HP, and RA stakeholders, while Zeng \textit{et al.} \cite{zeng2025planning} addressed multistakeholder conflicts through an equilibrium framework.
However, both neglect interactions between ReP2A and the gray ammonia markets.
When internal and external trading coexist, a two-level structure emerges in which electricity, hydrogen, carbon, and ammonia transactions mutually influence one another.
Designing an effective carbon transaction and allocation mechanism that enhances the competitiveness of green ammonia, limits emissions from gray ammonia production, and protects multistakeholder interests remains an open challenge.

\subsection{Research gaps and contributions}
\label{sec:contribution}
Based on the aforementioned literature review, despite the growing interest in green ammonia, several critical gaps remain:
\begin{enumerate}
	\item Existing studies focus on internal coordination within ReP2A systems, while neglecting competition with GA in coupled carbon and ammonia markets;
	
	\item 
	Carbon trading has not been integrated with the physical and economic interactions across power, hydrogen, and ammonia systems;
	
	\item  
	Incentive mechanisms that both improve green ammonia competitiveness and ensure IR among stakeholders are lacking.
	
\end{enumerate}

To fill these gaps, a hierarchical game-theoretical approach to characterize interactions between multistakeholder ReP2A and gray ammonia systems is developed in this paper. The main contributions are as follows:
\begin{enumerate}
	\item A two-level trading framework is proposed, in which CA and ammonia trading between ReP2A and GA form the outer level, while electricity and hydrogen trading among RG, HP, and RA form the inner level. The carbon market provides CA incentives for ReP2A production;
	
	\item  A hierarchical equilibrium model is formulated to capture cross-market interactions. The inner level is reformulated as a decomposable convex optimization problem \cite{zeng2025planning}, and the outer level is solved as a Nash-Cournot game via a mixed-integer linear program (MILP) derived from Karush-Kuhn-Tucker (KKT) conditions;
	
	\item An IR CA allocation mechanism (PCAM) is proposed for RG, HP, and RA stakeholders. Case studies show that improper allocation may reduce stakeholder profits, while the proposed mechanism preserves participation incentives.
\end{enumerate}


The remainder of this paper is organized as follows. Section~\ref{sec:structure} presents the system structure. Sections~\ref{sec:model} and \ref{sec:CAA} introduce the hierarchical game model and solution method, respectively. Section~\ref{sec:cases} reports simulation results, and Section~\ref{sec:conclusions} concludes this work.

\section{System structure and trading framework}
\label{sec:structure}

\begin{figure}[t]
	\centering
	\includegraphics[width=5.8in]{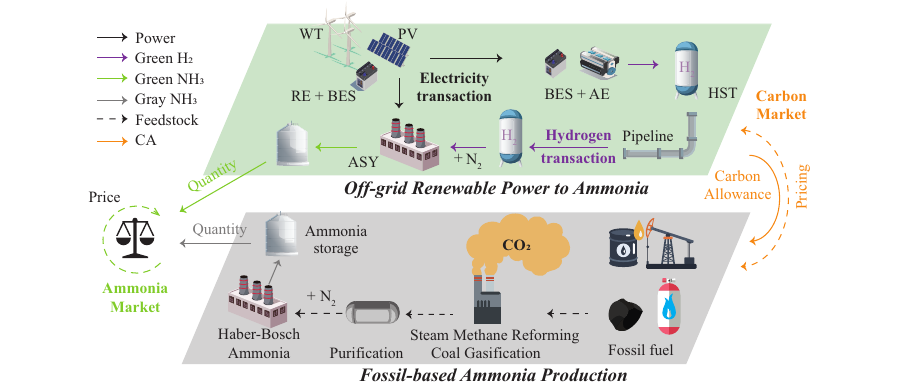}
	\caption{Structure and trading framework of ReP2A and gray ammonia systems}
	\label{fig:system}
\end{figure}

\subsection{System structure}

Fig.~\ref{fig:system} shows the integrated power-hydrogen-ammonia-carbon structure for ReP2A and fossil fuel-based gray ammonia systems.
ReP2A production follows a process chain consisting of renewable power generation, storage, transmission, power-to-hydrogen (P2H), hydrogen storage and delivery, and ammonia synthesis via the Haber-Bosch process \cite{zeng2025planning}.
In contrast, gray ammonia production relies on coal gasification or natural gas reforming, followed by hydrogen purification and ammonia synthesis.
We focus on off-grid ReP2A systems because they can satisfy green-certification requirements by complying with carbon-emission limits \cite{EU2023}, and thus, many real-world ReP2A projects are designed or operated in off-grid mode \cite{yu2024optimal, zhu2024exploring}.

The potential markets for green ammonia include the chemical-feedstock market and the shipping fuel market \cite{li2025redesigning,2025greenammonia}.
In the shipping-fuel market, green ammonia can command a green premium \cite{li2025redesigning,2025greenammonia}, i.e., enjoying a higher price, whereas gray ammonia cannot enter this market.
In contrast, in the chemical-feedstock market, there is currently no
green premium advantage; green and gray ammonia, being chemically identical, trade at the same price.
Therefore, this work focuses on the chemical-feedstock market, and our discussion follows on the following assumptions:

\begin{assumption}
	Green and gray ammonia enter the same chemical-feedstock market at the same price \footnote{Although green ammonia may command a ``green premium'' in emerging sectors (e.g., shipping), this paper focuses on the traditional chemical-feedstock market with identical pricing. Multi-market competition considering ``green premium'' in emerging sectors is left for future work.}.
\end{assumption}

\begin{assumption}
	Because ammonia has a limited economically viable transport range, the market is represented by one ReP2A producer and one GA producer competing locally \cite{ellwanger2023cost}.
\end{assumption}

\begin{assumption}
	Because ammonia logistics operate over long cycles, we assume that transactions settle weekly \cite{argus}.
\end{assumption}

\begin{figure}[t]
	\centering
	\includegraphics[width=5.3in]{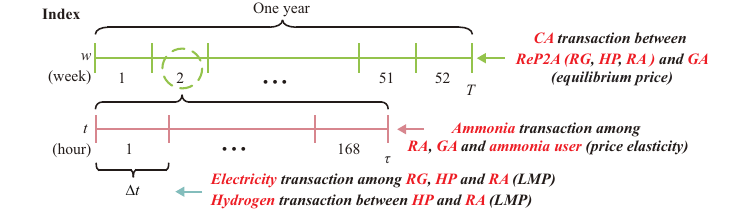}
	\caption{Time scales for carbon, ammonia, electricity, and hydrogen transactions}
	\label{fig:timescale}\vspace{-0pt}
\end{figure}

\subsection{Trading framework}

\subsubsection{Ammonia market}


In the ammonia market, RA and GA stakeholders sell to a representative user who engages in trading by choosing purchase quantities.
Their quantity decisions form Cournot competition \cite{Chen2021Conjectural, wu2023multi}, reflecting the imperfect relationship between price and demand.
The price elasticity is as follows:
\begin{align}
  \rho_w^{\text{am}}=&f(D_w^{\text{ga,sell}})=g(D_w^{\text{ra,sell}})
  =\rho^{\text{max}}-(D_w^{\text{ga,sell}}+D_w^{\text{ra,sell}})/k^{\text{am}}. \label{eq:amprice}
\end{align}

\subsubsection{Carbon market}

We incorporate a cap-and-trade scheme \cite{Chen2021Conjectural} into the carbon market to incentivize green ammonia production.
At the start of each year, CA is allocated to ReP2A and GA based on historical output, as described in Section~\ref{sec:setups}.
The CA granted to ReP2A, $q^{\text{rewa}}$, is then allocated among RG, HP, and RA, as detailed in Section~\ref{sec:CAAM}.
GA may buy CA from these entities to increase its permitted production under (\ref{eq:couple}), with total emissions less than $q^{\text{allo}}+q^{\text{rewa}}$. The annual market clears under (\ref{eq:marketc}) as follows:
\begin{gather}
q^{\text{rg}}+q^{\text{hp}}+q^{\text{ra}}\leq q^{\text{rewa}}, \label{eq:couple} \\
q^{\text{rg}}+q^{\text{hp}}+q^{\text{ra}}-q^{\text{ga}}=0:\rho^{\text{CA}}, \label{eq:marketc}
\end{gather}
where the dual variable $\rho^{\text{CA}}$ is the CA price, which is consistent across all entities.

\subsubsection{Electricity and hydrogen transactions}

Electricity and hydrogen trades occur within the ReP2A chain. The RG supplies electricity to the HP and RA, and the HP supplies hydrogen to the RA.
Settlement follows local marginal prices (LMPs) at each interval, and all transactions satisfy:
\begin{gather}
P_{t}^{\text{rg,sell,hp}}-P_t^{\text{hp,buy,rg}}=0:\rho_t^{\text{rg-hp,e}}, \label{eq:market1}\\
P_{t}^{\text{rg,sell,ra}}-P_t^{\text{ra,buy,rg}}=0:\rho_t^{\text{rg-ra,e}}, \label{eq:market2} \\
f_t^{\text{hp,sell,ra}}-f_t^{\text{ra,buy,hp}}=0:\rho_t^{\text{hp-ra,h}}. \label{eq:market3}
\end{gather}

All trading entities and their time scales are summarized in Fig.~\ref{fig:timescale} for ease of understanding.


\section{Operational decision modeling of ReP2A and gray ammonia systems}
\label{sec:model}

This section presents a hierarchical game framework, which is based on noncooperative game theory and captures interactions between ReP2A and GA, as well as among stakeholders in the ReP2A chain. The mathematical formulations are given below.

\subsection{Hierarchical game framework}
\label{sec:hiergame}

The hierarchical structure is illustrated in Fig.~\ref{fig:game}.
At the inner level (inside the ReP2A chain), RG, HP, and RA engage in electricity trading, hydrogen trading, and CA allocation, forming a Nash equilibrium.
At the outer level, the ReP2A chain and GA producer participate in the carbon and ammonia markets, captured by Nash-Cournot equilibrium.

\begin{figure}[t]
	\centering
	\includegraphics[width=5.7in]{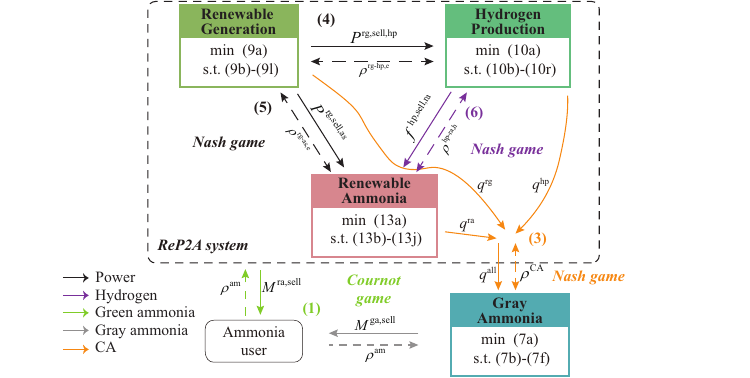}
	\caption{Hierarchical game framework inside the ReP2A chain and between ReP2A and GA systems}
	\label{fig:game}
\end{figure}

\subsection{Operation model of the gray ammonia system}
\label{sec:outgame}

The GA stakeholder is modeled as a profit-seeking producer subject to chemical process and CA constraints as follows:
\begin{subequations}
	\begin{align}
		\min~~ &C_{\text{ga}}=\Delta t \sum_{t=1}\nolimits^T c^{\text{ga}}M_t^{\text{ga,pro}} - q^{\text{ga}} \rho^{\text{CA}}-\sum\nolimits_w f(D_w^{\text{ga,sell}})D_w^{\text{ga,sell}}, \\
		\text{s.t.}~~
		&\text{ASY operation constraints in the GA:} \nonumber\\
		&\underline{\eta}^{\text{asy}}W^{\text{ga,asy}}\leq M_{t}^{\text{ga,pro}}\leq\overline{\eta}^{\text{asy}}W^{\text{ga,asy}}, \label{eq:aslimit}\\
		&-\underline{r}^{\text{asy}}W^{\text{ga,asy}}\leq M_{t+1}^{\text{ga,pro}}-M_{t}^{\text{ra,pro}}\leq\overline{r}^{\text{asy}}W^{\text{ga,asy}}, \label{eq:asramp}    \\
		&\textstyle \sum_{t=(w-1)\tau}^{w\tau} M_{t}^{\text{ga,pro}} \Delta t=D_w^{\text{ga,sell}}, \label{eq:assell}\\
		&\text{Carbon allowance compliance and settlement:} \nonumber\\
		&q_t^{\text{emis}}= k^{\text{emis}} M_{t}^{\text{ga,pro}} \Delta t, \label{eq:asemis}\\
		&\textstyle \sum_{t=0}^{T} q_t^{\text{emis}}\leq q^{\text{ga}}+q^{\text{allo}}, \label{eq:GACA}
	\end{align}
\end{subequations}
where overall cost $C_{\text{ga}}$ includes production cost, CA purchases, and ammonia sales revenue; (\ref{eq:aslimit})--(\ref{eq:asramp}) enforce the load range and ramping limits \cite{li2025redesigning}; (\ref{eq:assell}) links production and sales; and (\ref{eq:asemis})--(\ref{eq:GACA}) specify carbon emissions and CA compliance, i.e., that carbon emissions from GA must not exceed its CA.
Given that GA output is controllable and weekly ammonia demand in the spot market is limited, storing ammonia does not provide additional profit. Therefore, ammonia storage tanks (ASTs) are not included in the model.

For conciseness, the operation model of GA is expressed as:
\begin{align}
  \nonumber \min_{\bm{x}_{\text{ga}}}~ &\bm{C}_{\text{ga,1}}^{\text{T}}\bm{x}_{\text{ga}} + \bm{C}_{\text{ga,2}}^{\text{T}} \bm{x}_{\text{ga}}^{2},\\
\nonumber  \text{s.t.} ~~&\bm{A}_{\text{ga,1}}\bm{x}_{\text{ga}}=\bm{B}_{\text{ga,1}}:\bm{\lambda}_{\text{ga}},
  \\&\bm{A}_{\text{ga,2}}\bm{x}_{\text{ga}}\geq \bm{B}_{\text{ga,2}}:\bm{\mu}_{\text{ga}}, \label{eq:GApro}
\end{align}
where $\bm{x}_{\text{ga}}$ comprises all decision variables;
$\bm{A}_\cdot$, $\bm{B}_\cdot$, and $\bm{C}_\cdot$ are coefficient matrices;
$\bm{\lambda}_{\cdot}$ and $\bm{\mu}_{\cdot}$ are dual variables; and
$\bm{C}_{\text{ga,2}}$ is positive definite. Thus, problem~(\ref{eq:GApro}) is a convex quadratic programming (QP) problem.

\subsection{Operation model of the multistakeholder ReP2A system}
\label{sec:ingame}

The ReP2A system involves a multistakeholder chain consisting of RG, HP, and RA stakeholders, each acting with individual decision-making
\cite{zeng2025planning}. Their decision models are given below.

\subsubsection{Renewable power generation}

The RG stakeholder minimizes its cost $C_{\text{rg}}$ by managing power generation, storage, and sales and CA trades as follows:
\begin{subequations}\label{eq:RG}
\begin{align}
	\min~~ &C_{\text{rg}}=\Delta t \sum\nolimits_{t=1}^T\big[(-P_{t}^{\text{rg,sell,hp}}\rho_t^{\text{rg-hp,e}}-P_{t}^{\text{rg,sell,as}}\rho_t^{\text{rg-as,e}}) +\sigma^{\text{deg}}P_{t}^{\text{rg,bes,d}}\big] - q^{\text{rg}} \rho^{\text{CA}}, \label{eq:objrg}\\
\text{s.t.}~~
    &\text{Renewable power availability:} \nonumber\\
    & P_{t}^{\text{rg,wt/pv}}=P_{t}^{\text{rg,wt/pv,max}}-P_{t}^{\text{rg,wt/pv,curt}}, \ P_{t}^{\text{rg,wt/pv,curt}}\ge 0, \label{eq:wtpv1}\\
    & \text{Operation constraints of battery energy storage (BES):} \nonumber\\
    &\big|P_{t}^{\text{rg,bes,c/d}}\pm Q_{t}^{\text{rg,bes}}\big| \leq \sqrt{2} W^{\text{rg,bes}},~\big|Q_{t}^{\text{rg,bes}}\big| \leq W^{\text{rg,bes}},  \label{eq:bes1}\\
    &\bm{0} \leq \big[ P_{t}^{\text{rg,bes,c}},P_{t}^{\text{rg,bes,d}} \big] \leq \big[ 0.5W^{\text{rg,bes}},0.5W^{\text{rg,bes}} \big],  \label{eq:bes2}  \\
    & S_{t}^{\text{rg,bes}}=(1-\zeta^{\text{bes}})S_{t-1}^{\text{rg,bes}} + (\eta^{\text{bes,c}}P_{t}^{\text{rg,bes,c}} - \dfrac {P_{t}^{\text{rg,bes,d}}}{\eta^{\text{bes,d}}})\Delta t, \label{eq:bes3}\\
    & S_{t=(w-1)\tau}^{\text{rg,bes}}=S_{t=w\tau}^{\text{rg,bes}}, \forall w, \\
    &\underline{\eta}^{\text{bes}}W^{\text{rg,bes}} \leq S_{t}^{\text{rg,bes}}\leq\overline{\eta}^{\text{bes}}W^{\text{rg,bes}}, \label{eq:bes5}\\
    & \text{Electrical network power flow:} \nonumber\\
    & \big[| P_{t}^{\text{rg,wt/pv}} \pm Q_{t}^{\text{rg,wt/pv}}|, |Q_{t}^{\text{rg,wt/pv}}|\big] \leq [\sqrt{2},1] W^{\text{rg,wt/pv}}, \label{eq:wtpv2}\\
    & \sum\nolimits_{j':j\to j'}P_{jj',t}=P_{j,t}^{\text{rg,wt}}+P_{j,t}^{\text{rg,pv}}+P_{j,t}^{\text{rg,bes,d}}-P_{j,t}^{\text{rg,bes,c}} \nonumber\\
    &~~~~~~~~~-P_{j,t}^{\text{rg,sell,hp}}-P_{j,t}^{\text{rg,sell,as}}+  \sum\nolimits_{i:i\to j}P_{ij,t}-r_{ij}\ell_{ij,t}, \label{eq:dist1}\\
	& \sum\nolimits_{j':j\to j'}Q_{jj',t}=Q_{i,t}^{\text{rg,wt}}+Q_{i,t}^{\text{rg,pv}}+Q_{i,t}^{\text{rg,vc}}+Q_{i,t}^{\text{rg,bes}}+ \sum\nolimits_{i:i\to j}Q_{ij,t}-x_{ij}\ell_{ij,t}, \label{eq:dist2}\\
	&\upsilon_{j,t}=\upsilon_{i,t}-2r_{ij}P_{ij,t}-2x_{ij}Q_{ij,t}, \label{eq:dist3}\\
	&\underline{\upsilon}_{i} \leq \upsilon_{i,t}\leq\overline{\upsilon}_{i},  \label{eq:dist4}
    \end{align}
\end{subequations}
\noindent
where $C_{\text{rg}}$ includes electricity revenues, CA sales, and BES degradation cost \cite{yu2024optimal}; the degradation cost allows the charging/discharging complementarity constraint to be convexly relaxed \cite{wang2024exact};
(\ref{eq:bes1}) and (\ref{eq:bes2}) are active and reactive power constraints of the BES; the state of charge is limited by (\ref{eq:bes3})--(\ref{eq:bes5}); and (\ref{eq:wtpv1})--(\ref{eq:wtpv2}) constrain the active and reactive power of WT/PV.
The network power flow is described by the LinDistFlow model because of its radial topology \cite{zeng2024scheduling}, which includes branch power and voltage balance (\ref{eq:dist1})--(\ref{eq:dist3}) and nodal voltage limits (\ref{eq:dist4}).

\subsubsection{Hydrogen production}

The HP stakeholder purchases electricity from RG, produces hydrogen, sells hydrogen to the RA stakeholder, and trades CA to minimize cost $C_{\text{hp}}$, following:
\begin{subequations}
\begin{align}
	\min~~ \nonumber   &C_{\text{hp}}=\sum\nolimits_{t=1}^T \big[\big(P_t^{\text{hp,buy,rg}}\rho_t^{\text{rg-hp,e}}+\gamma\textstyle\sum_{mn}(p_{m,t}-p_{n,t}) \big) \\
	&~~~~~~~~~+\sigma^{\text{deg}}P_{t}^{\text{hp,bes,d}}-f_t^{\text{hp,sell,ra}}\rho_t^{\text{hp-ra,h}}\big]\Delta t -q^{\text{hp}} \rho^{\text{CA}}, \label{eq:objhp}\\
	\text{s.t.}~~
    &\text{P2H Operation constraints via water electrolysis:} \nonumber\\
    &f_{t}^{\text{hp,pro}}=P_{t}^{\text{hp,ae}}\eta^{\text{p2h}}, \label{eq:p2h}\\
	&\underline{\eta}^{\text{ae}}W^{\text{hp,ae}}\leq P_{t}^{\text{hp,ae}}\leq\overline{\eta}^{\text{ae}}W^{\text{hp,ae}},  \label{eq:p2hmaxmin}\\
	&P_{t}^{\text{hp,comp}}=f_{t}^{\text{hp,pro}}\eta^{\text{comp}}, \label{eq:hcomp}\\
	&P_{t}^{\text{hp,buy,rg}}+P_{t}^{\text{hp,bes,d}}=P_{t}^{\text{hp,bes,c}}+P_{t}^{\text{hp,ae}}+P_{t}^{\text{hp,comp}}. \label{eq:hbanlance}\\
    &\text{Operation limits of BES equipped in HP, which have the} \nonumber \\
    & \text{same forms of (\ref{eq:bes1})--(\ref{eq:bes5}) for }\{P_{t}^{\text{hp,bes,c/d}},Q_{t}^{\text{hp,bes}},S_{t}^{\text{hp,bes}}\}. \hspace{-3pt}\label{eq:hpbes}\\
    &\text{HST operation constraints:} \nonumber\\
    &S_{t+1}^{\text{hp,hst}} =S_{t}^{\text{hp,hst}}+(f_{t}^{\text{hp,hst,in}}-f_{t}^{\text{hp,hst,out}})\Delta t, \label{eq:hes1}\\
	&S_{t=(w-1)\tau}^{\text{hp,hst}}=S_{t=w\tau}^{\text{hp,hst}}, \forall w \\
	&\underline{\eta}^{\text{h}}W^{\text{hp,hst}}\leq S_{t}^{\text{hp,hst}}\leq\overline{\eta}^{\text{h}}W^{\text{hp,hst}}, \label{eq:hes3} \\
	&\textbf{0}\leq [f_{t}^{\text{hp,hst,in}},f_{t}^{\text{hp,hst,out}}]\leq[0.5W^{\text{hp,hst}}, 0.5W^{\text{hp,hst}}].    \label{eq:hes4}\\
    &\text{Hydrogen delivery via pipelines:} \nonumber\\
    &({F_{mn,t}}/{K_{mn}^{\text{gf}}})^{2}\leq p_{m,t}^{2}-p_{n,t}^{2}, \label{eq:weymouth}\\
    &{F_{mn,t}}/{K_{mn}^{\text{gf}}}>p_{m,t}-p_{n,t}, \label{eq:weymouth2}\\
	&F_{mn,t}=(F_{mn,t}^{\text{in}}+F_{mn,t}^{\text{out}})/2, ~F_{mn,t}\geq0 ,  \label{eq:hydrogenflow}\\
	&LP_{mn,t}=K_{mn}^{\text{lp}}(p_{m,t}+p_{n,t})/2, \label{eq:lp}\\
	&LP_{mn,t+1}=LP_{mn,t}+F_{mn,t}^{\text{in}}-F_{mn,t}^{\text{out}}, \label{eq:linepack}\\
	&\underline{p}_{m}\leq p_{m,t}\leq\overline{p}_{m},  \label{eq:pressure}\\
	&LP_{mn,t=(w-1)\tau}=LP_{mn,t=w\tau}, \forall w, \label{eq:lp-balance}\\
 &f_{t}^{\text{hp,pro}}+f_{t}^{\text{hp,hst,out}}-f_{t}^{\text{hp,hst,in}}+F_{mn,t}^{\text{out}}-F_{mn,t}^{\text{in}}=f_{t}^{\text{hp,sell,ra}}. \label{eq:hydrogenbalance}
\end{align}
\end{subequations}
\noindent
where $C_{\text{hp}}$ consists of electricity purchases, hydrogen sales, CA sales, BES degradation, and a nodal pressure penalty to ensure exact relaxation of the Weymouth equation \cite{liu2020application};
(\ref{eq:p2h})--(\ref{eq:p2hmaxmin}) describe the efficiency and load range of HP; 
(\ref{eq:hcomp})--(\ref{eq:hbanlance}) specify compressor load and plant power balance.
The charging/releasing behaviors of BES and HST follow (\ref{eq:hpbes})--(\ref{eq:hes4}), and
pipeline physics follow the 
Weymouth equation, with its second-order cone relaxation given in (\ref{eq:weymouth})--(\ref{eq:hydrogenflow}). Eqs. (\ref{eq:lp})--(\ref{eq:lp-balance}) establish the relation among pressure, hydrogen flow, and linepack, and hydrogen balance is enforced by (\ref{eq:hydrogenbalance}).

The HP operation problem is a second-order cone programming (SOCP).
To improve tractability, the conic constraint~(\ref{eq:weymouth}) is approximated by polyhedral linear constraints as follows:
\begin{subequations}\label{eq:soc}
	\begin{gather}
		\xi^0\geq F_{mn}/K_{mn}^\text{gf},\omega^0\geq p_n, \\
		\xi^Z\leq p_m,\omega^Z\leq\tan(\frac{\pi}{2^{Z+1}})\xi^Z, \\
		\xi^z=\sin(\frac{\pi}{2^{z+1}})\omega^{z-1}+\cos(\frac{\pi}{2^{z+1}})\xi^{z-1},\forall z, \\
		\omega^z\geq\left|\cos(\frac{\pi}{2^{z+1}})\omega^{z-1}-\sin(\frac{\pi}{2^{z+1}})\xi^{z-1}\right|,\forall z,
	\end{gather}
\end{subequations}
where $(\xi^z)_{z \in [0,...,Z]}$ and $\omega^z$ are auxiliary variables.  This transforms the SOCP into a linear programming (LP) problem. The approximation error is quantified by
\begin{subequations}
	\begin{gather}
		({F_{mn,t}}/{K_{mn}^{\text{gf}}})^{2}\leq [1+\epsilon(Z)]^2 p_{m,t}^{2}-p_{n,t}^{2}, \\
		\epsilon(Z)={1}/{\cos({\frac{\pi}{2^{Z+1}}})}-1.
	\end{gather}
\end{subequations}
The larger the value of $Z$, the smaller the error, reaching about 0.05\% at $Z=6$.

\subsubsection{Green ammonia synthesis}

The RA stakeholder determines the purchase of feedstock hydrogen and electricity and production, storage, and sales of green ammonia, as well as trades of CA, as follows:
\begin{subequations}\label{eq:RA}
\begin{align}
	\min~~ \nonumber  & C^{\text{ra}}=\sum_{t=1}\nolimits^T \big[(f_t^{\text{ra,buy,hp}}\rho_t^{\text{hp-ra,h}}+P_t^{\text{ra,buy,rg}}\rho_t^{\text{rg-ra,e}})\Delta t\\
	&~~~~~~~~
	+P_t^{\text{ra,back}}\rho_t^{\text{ra,back}}\Delta t \big]-q^{\text{ra}} \rho^{\text{CA}}-\sum_w g(M_w^{\text{ra,sell}}) M_w^{\text{ra,sell}}  \label{eq:objra}\\
\text{s.t.}~~
    &\text{operation constraints of HST equipped in RA, the same} \nonumber\\
    &\text{form as (\ref{eq:hes1})--(\ref{eq:hes4}) for }\{f_{t}^{\text{ra,hst,in/out}},S_{t}^{\text{ra,hst}},W^{\text{ra,hst}}\}, \label{eq:rahst}\\
    &\text{ASY operation constraints:} \nonumber\\
    &f_{t}^{\text{ra,use}}+f_{t}^{\text{ra,hst,in}}=f_{t}^{\text{ra,hst,out}}+f_{t}^{\text{ra,buy,hp}},  \label{eq:rah}\\
    &P_{t}^{\text{ra,back}}+P_{t}^{\text{ra,buy,rg}}=P_{t}^{\text{as,asy}},  \label{eq:pas}\\
    &M_{t}^{\text{ra,pro}}=f_{t}^{\text{ra,use}}\eta^{\text{h2a}},  \label{eq:h2a}\\
	&M_{t}^{\text{ra,pro}}=P_{t}^{\text{ra,asy}}\eta^{\text{p2a}},  \label{eq:p2a}\\
    &\text{and ramping limits as (\ref{eq:aslimit})--(\ref{eq:asramp}) for } \{M_{t}^{\text{ra,pro}},W^{\text{ra,asy}}\}, \label{eq:ras}\\
    &\text{AST operation constraints:} \nonumber\\
    &S_{w+1}^{{\text{ra,ast}}}=S_{w}^{{\text{ra,ast}}}+\textstyle \sum_{t=(w-1)\tau+1}^{w\tau} M_{t}^{{\text{ra,pro}}}\Delta t-D_{w}^{{\text{ra,sell}}}, \label{eq:ast1}\\
	&S_{w=0}^{{\text{ra,ast}}}=S_{w={T}/{\tau}+1}^{{\text{ra,ast}}},\\
	&0\leq S_{w}^{{\text{ra,ast}}}\leq W^{{\text{ra,ast}}} \label{eq:ast3},
\end{align}
\end{subequations}
\noindent
where cost $C^{\text{ra}}$ consists of electricity and hydrogen purchases, the backup power cost, and ammonia and CA sales; 
(\ref{eq:rah})--(\ref{eq:pas}) specify hydrogen and power balance at the green ammonia plant, with hydrogen and power consumption shown in (\ref{eq:h2a})--(\ref{eq:p2a}), respectively; and
(\ref{eq:ast1})--(\ref{eq:ast3}) are AST operational constraints, which optimize trading by adjusting storage levels.


\section{Equilibrium analysis of the hierarchical incentive mechanism}
\label{sec:CAA}

In this section, the solution method for the hierarchical incentive mechanism
is developed. The inner-level Nash equilibrium is converted into a convex optimization via KKT conditions.
To improve tractability, the inner equilibrium is decomposed into several ammonia production subproblems (SPs) and one main transaction problem (MP).
The hierarchical game is then solved based on the KKT system. The overall solution framework is shown in Fig.~\ref{fig:solution}.
The results of the inner-level analysis reveal that the CA allocation within the ReP2A system is not uniquely determined at equilibrium. To ensure that all entities in the ReP2A system benefit, a CA allocation mechanism (CAM) that considers participation willingness is introduced.

\begin{figure}[t]
	\centering
	\includegraphics[width=6.5in]{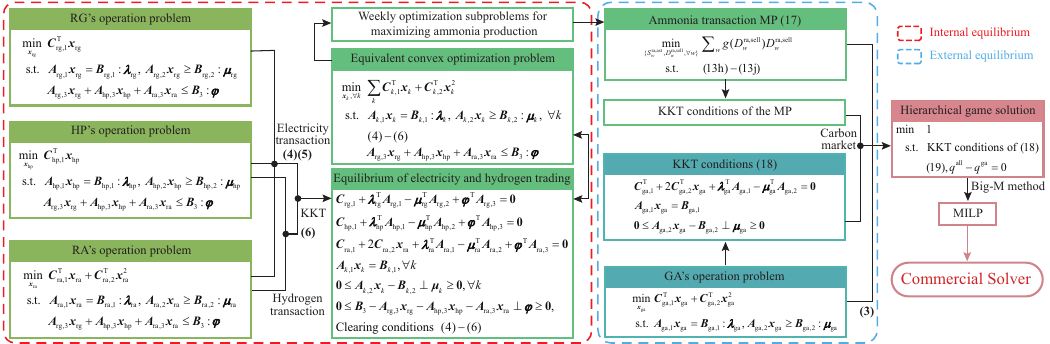}
	\caption{Overall solution framework for the equilibrium analysis of the hierarchical game}
	\label{fig:solution}
\end{figure}

\subsection{Inner-level equilibrium among RG, HP, and RA}
\label{sec:why}

For conciseness, the operation problems of RG (Eqs. (\ref{eq:couple}) and (\ref{eq:RG})), HP (Eqs. (\ref{eq:couple}), (\ref{eq:objhp})--(\ref{eq:hes4}), (\ref{eq:weymouth2})--(\ref{eq:hydrogenbalance}), and (\ref{eq:soc})), and RA (Eqs. (\ref{eq:couple}) and (\ref{eq:RA})) are compactly written as
follows:
\begin{align}
  \nonumber \min_{\bm{x}_{k}}~ &\bm{C}_{k\text{,1}}^{\text{T}}\bm{x}_{k} + \bm{C}_{k\text{,2}}^{\text{T}} \bm{x}_{k}^{2},\\
  \nonumber  \text{s.t.} ~~&\bm{A}_{k\text{,1}}\bm{x}_{k}=\bm{B}_{k\text{,1}}:\bm{\lambda}_{k},\\
  \nonumber  &\bm{A}_{k\text{,2}}\bm{x}_{k}\geq \bm{B}_{k\text{,2}}:\bm{\mu}_{k},\\
  ~~&\bm{A}_{\text{rg,3}}\bm{x}_\text{rg}+\bm{A}_{\text{hp,3}}\bm{x}_\text{hp}+\bm{A}_{\text{ra,3}}\bm{x}_\text{ra}\leq \bm{B}_{\text{3}}:\bm{\varphi}, \label{eq:stakeholder}
\end{align}
where $(\bm{x}_k)_{k \in \text{\{rg,hp,ra\}}}$ denote the decision variables of RG, HP, and RA, respectively; $\bm{\varphi}$ is the dual variable; $\bm{C}_{\text{rg/hp,2}}$ are both zero; and $\bm{C}_{\text{ra,2}}$ is positive definite. Since all stakeholders share identical dual variables for the coupling constraints in (\ref{eq:stakeholder}), the equilibrium constitutes a \textit{variational equilibrium (VE)}, which is unique \cite{rosen1965existence}.

Because problem~(\ref{eq:stakeholder}) is either an LP  or a QP problem, which are both convex, the equilibrium in the multistakeholder ReP2A system can be obtained from the joint KKT conditions \cite{chen2020operational, guo2021equilibrium}: 
\begin{subequations}\label{eq:allkkt}
\begin{align}
  &\bm{C}_{\text{rg,1}}+\bm{\lambda}_{\text{rg}}^{\text{T}}\bm{A}_{\text{rg,1}}-\bm{\mu}_{\text{rg}}^{\text{T}}\bm{A}_{\text{rg,2}}+\bm{\varphi}^{\text{T}}\bm{A}_{\text{rg,3}}=\bm{0},\\
&\bm{C}_{\text{hp,1}}+\bm{\lambda}_{\text{hp}}^{\text{T}}\bm{A}_{\text{hp,1}}-\bm{\mu}_{\text{hp}}^{\text{T}}\bm{A}_{\text{hp,2}}+\bm{\varphi}^{\text{T}}\bm{A}_{\text{hp,3}}=\bm{0},\\
&\bm{C}_{\text{ra,1}}+ 2\bm{C}_{\text{ra,2}}\bm{x}_\text{ra}+\bm{\lambda}_{\text{ra}}^{\text{T}}\bm{A}_{\text{ra,1}}-\bm{\mu}_{\text{ra}}^{\text{T}}\bm{A}_{\text{ra,2}}+\bm{\varphi}^{\text{T}}\bm{A}_{\text{ra,3}}=\bm{0},\\
&\bm{A}_{k\text{,1}}\bm{x}_{k}=\bm{B}_{k\text{,1}},~ \forall k\\
&\bm{0}\leq\bm{A}_{k\text{,2}}\bm{x}_{k}-\bm{B}_{k\text{,2}}\perp\bm{\mu}_{k}\geq \bm{0}, ~ \forall k\\
&\bm{0}\leq \bm{B}_{\text{3}}-\bm{A}_{\text{rg,3}}\bm{x}_\text{rg}-\bm{A}_{\text{hp,3}}\bm{x}_\text{hp}-\bm{A}_{\text{ra,3}}\bm{x}_\text{ra}\perp \bm{\varphi}\geq \bm{0},\\
&\text{Clearing conditions (\ref{eq:market1})--(\ref{eq:market3}).}
\end{align}
\end{subequations}

Next, an equivalent convex optimization (\ref{eq:problem}), which is
a positive definite QP, 
is employed to obtain the equilibrium \cite{zeng2025planning, egging2020solving}, as its KKT conditions are consistent with 
(\ref{eq:allkkt}).
Its derivation is straightforward, and we thus skip it 
for brevity. \vspace{-0pt}
\begin{subequations}\label{eq:problem}
\begin{align}
\min_{\bm{x}_{k}, \forall k}~ & \sum_k \bm{C}_{k\text{,1}}^{\text{T}}\bm{x}_{k} + \bm{C}_{\text{ra,2}}^{\text{T}} \bm{x}_\text{ra}^{2}, \label{eq:objall}\\
 \text{s.t.} ~~&\bm{A}_{k\text{,1}}\bm{x}_{k}=\bm{B}_{k\text{,1}}:\bm{\lambda}_{k},~\forall k\\
 &\bm{A}_{k\text{,2}}\bm{x}_{k}\geq \bm{B}_{k\text{,2}}:\bm{\mu}_{k}, ~\forall k\\
~~&\bm{A}_{\text{rg,3}}\bm{x}_\text{rg}+\bm{A}_{\text{hp,3}}\bm{x}_\text{hp}+\bm{A}_{\text{ra,3}}\bm{x}_\text{ra}\leq \bm{B}_{\text{3}}:\bm{\varphi},\\
~~&\text{(\ref{eq:market1})--(\ref{eq:market3})},
\end{align}\end{subequations}
where (\ref{eq:objall}) is obtained by summing (\ref{eq:objrg})--(\ref{eq:objhp}) and (\ref{eq:objra}), representing the total cost of RG, HP, and RA excluding terms related to electricity and hydrogen transactions, i.e., $P_{t}^{\text{hp,buy,rg}}\rho_t^{\text{rg-hp,e}}-P_{t}^{\text{rg,sell,hp}}\rho_t^{\text{rg-hp,e}}$, $P_{t}^{\text{ra,buy,rg}}\rho_t^{\text{rg-ra,e}}-P_{t}^{\text{rg,sell,ra}}\rho_t^{\text{rg-ra,e}}$, and
$f_{t}^{\text{ra,buy,hp}}\rho_t^{\text{hp-ra,h}}-P_{t}^{\text{hp,sell,ra}}\rho_t^{\text{hp-ra,h}}$.

The following proposition assists in the solution of the overall hierarchical game and is used in Section \ref{sec:method}.
\begin{proposition}
At equilibrium, the total traded CA, $\hat{q}^\text{all}$, is uniquely determined, whereas the CA allocations $q^{\text{rg}}$, $q^{\text{hp}}$, and $q^{\text{ra}}$ are indeterminate.
\end{proposition}

\begin{proof}
The CA constraint in (\ref{eq:problem}) follows (\ref{eq:couple}), and the objective includes the terms of $q^{\text{rg}} \rho^{\text{CA}}$, $q^{\text{hp}} \rho^{\text{CA}}$, and $q^{\text{ra}} \rho^{\text{CA}}$.
We use $q^\text{all}$ to replace $q^\text{rg/hp/ra}$, variables related to CA allocation, with:
\begin{align}
q^\text{all}=q^{\text{rg}}+q^{\text{hp}}+q^{\text{ra}}. \label{eq:replace}
\end{align}

This substitution does not alter the solution to (\ref{eq:problem}) and yields the optimal $q^\text{all}$. Because $q^{\text{rg}}$, $q^{\text{hp}}$, and $q^{\text{ra}}$ are redundant variables in (\ref{eq:problem}), they are indeterminate at equilibrium.
\end{proof}


\subsection{Equilibrium solution for the hierarchical game}
\label{sec:method}

\begin{figure}[t]
	\centering
	\includegraphics[width=3.8in]{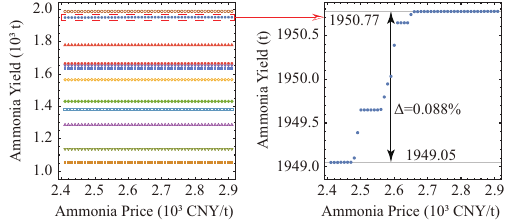}
	\caption{Ammonia yield for 12 weeks under varying ammonia prices.}
	\label{fig:proof}
\end{figure}

By replacing the inner equilibrium with (\ref{eq:problem}), one can theoretically solve the carbon and ammonia market equilibrium by jointly handling the KKT conditions of (\ref{eq:problem}) and (\ref{eq:GApro}). However, owing to long-term trading, (\ref{eq:problem}) is very large. Even if 12 typical weeks are adopted to represent a year, the problem contains 424,404 constraints and 264,102 variables. Its optimality, whether formulated as KKT (including complementarity slackness) or primal-dual (including strong-duality equality) conditions, introduces bilinear and binary terms, making it computationally intractable. Thus, simplifications 
are required.


\begin{remark}
	Physically, (\ref{eq:problem}) schedules the use of renewable energy sources (RES) to maximize revenues from ammonia and CA transactions.
	Because ammonia trading occurs weekly, AST operations can be decoupled from ammonia production.
\end{remark}

\begin{remark}
	The weekly ammonia production remains nearly invariant under changes in the price of ammonia. This is verified using over 60,000 simulations, each across 12 weeks.
	The results 
	shown in Fig.~\ref{fig:proof} reveal that the ammonia yield fluctuates by less than 0.1\%. Thus, the scheduling of ammonia production can be decoupled from ammonia trading. 
\end{remark}

Based on the above, the inner equilibrium problem (\ref{eq:problem}) is decomposed into 12 production SPs and one trading MP.
By solving the SPs to obtain the weekly ammonia yield $\Delta t \sum_{t=(w-1)\tau+1}^{w\tau} \hat{M}_{t}^{{\text{ra,pro}}}$,
and substituting it into (\ref{eq:ast1}) as parameters, the MP can be simplified as follows:
\begin{subequations}\label{eq:MP}
\begin{align}
&\min_{\{S_{w}^{{\text{ra,ast}}},D_w^{\text{ra,sell}},\forall w\}}~\sum\nolimits_w g(D_w^{\text{ra,sell}}) D_w^{\text{ra,sell}},\\
&\text{~~~~~~~~~s.t. (\ref{eq:ast1})--(\ref{eq:ast3})}.
\end{align}
\end{subequations}


Finally, we can derive the KKT conditions of the hierarchical game. The KKT conditions of GA operation (\ref{eq:GApro}) include:
\begin{subequations}\label{eq:gakkt}
\begin{gather}
\bm{C}_{\text{ga,1}}+ 2\bm{C}_{\text{ga,2}}^{\text{T}} \bm{x}_{\text{ga}}+\bm{\lambda}_{\text{ga}}^{\text{T}}\bm{A}_{\text{ga,1}}-\bm{\mu}_{\text{ga}}^{\text{T}}\bm{A}_{\text{ga,2}}=\bm{0}, \label{eq:gakkt1}\\
\bm{A}_{\text{ga,1}}\bm{x}_{\text{ga}}=\bm{B}_{\text{ga,1}}, \label{eq:gakkt2}\\
\bm{0}\leq\bm{A}_{\text{ga,2}}\bm{x}_{\text{ga}}-\bm{B}_{\text{ga,2}}\perp\bm{\mu}_{\text{ga}}\geq \bm{0}, \label{eq:gakkt3}
\end{gather}
\end{subequations}
where (\ref{eq:gakkt1}) represents stationarity, (\ref{eq:gakkt2}) summarizes the equality constraints, and (\ref{eq:gakkt3}) aggregates the inequality constraints with complementary slackness.
Similarly, the KKT conditions of the ReP2A operation (\ref{eq:MP}) are derived. 

By combining (\ref{eq:gakkt}), the KKT conditions of (\ref{eq:MP}), and the clearing condition of CA transaction, we obtain the outer-level market equilibrium.
For computational efficiency, these equations are replaced by the following  optimization  problem: \vspace{0pt}
\begin{subequations}\label{eq:probfinal}
\begin{align}
 &\min~ 1,\\
 &\text{ s.t.~} \text{(\ref{eq:gakkt})},~\text{KKT conditions of (\ref{eq:MP})},~q^{\text{all}}=q^{\text{ga}},
\end{align}
\end{subequations}
in which the complementary slackness can be reformulated into a mixed-integer linear form by the big-M method, resulting in an MILP that is easy to solve with commercial software.

After the outer equilibrium is solved, the resulting ammonia price $\rho_w^{\text{am}}$ and CA price $\rho^{\text{CA}}$, together with the AST operational results, are substituted into the inner equilibrium problem (\ref{eq:problem}) to obtain electricity prices $\rho^{\text{rg-hp/ra,e}}$ and hydrogen prices $\rho_t^{\text{hp-ra,h}}$. By far, all the variables are obtained at equilibrium. Discussion on the uniqueness of the equilibrium is provided in Section \ref{sec:solutionstability}.

\subsection{An Individually Rational CA allocation mechanism}
\label{sec:CAAM}

Once equilibrium is reached, the total traded quantity $\hat{q}^\text{all}$ and clearing price $\hat{\rho}^\text{CA}$ of CA are determined. The remaining task is to allocate the resulting carbon revenue among RG, HP, and RA stakeholders in a fair and individually rational manner. All stakeholders remain individually rational, meaning that their revenue with CA trading cannot fall below the revenue without such trading; otherwise, CA rewards would not motivate participation.

To achieve this, carbon revenue is allocated by minimizing the aggregate deviation in revenue changes across stakeholders:
\begin{align}
	\hspace{-8pt}  \Delta J_\text{sum}=|\Delta J_\text{rg}-\Delta J_\text{hp}|+|\Delta J_\text{rg}-\Delta J_\text{ra}|+|\Delta J_\text{hp}-\Delta J_\text{ra}|, \label{eq:Dsum}
\end{align}
where $(\Delta J_k)_{k \in \text{\{rg,hp,ra\}}}$ are 
revenue changes of RG/HP/RA; and 
\begin{align}
	\Delta J_k=(\hat{C}_k-\tilde{C}_k)/\tilde{C}_k, ~\Delta J_k\geq0,~\forall k \in \text{\{rg,hp,ra\}}. \label{eq:Dk}
\end{align}
Here, $\tilde{C}_k$ is the revenue under the equilibrium without CA trading (i.e., $q^k\rho^{\text{CA}}=0$), and $\hat{C}_k$ is the revenue under CA trading, including the carbon revenue to be allocated.

The CA allocation problem is therefore given as follows:
\begin{subequations}\begin{align}
		&\min_{q^k,\forall k}~ \Delta J_\text{sum},\\
		&\text{~~s.t.~} \hat{q}^\text{all}=q^{\text{rg}}+q^{\text{hp}}+q^{\text{ra}}.
	\end{align}
\end{subequations}

As CA rewards are properly allocated, this mechanism ensures that all entities retain sufficient incentives to participate in ReP2A production, offering a practical approach to balancing benefits in the multistakeholder system.
The proposed CAM (PCAM) ensures that the IR in the multistakeholder ReP2A system, meaning that the carbon trading strategy given by the equilibrium model is in the best interest of each entity.
Notably, strategic misreporting is practically infeasible in this context, as stakeholders' energy-mass flows are physically metered and subject to strict verification by regulatory authorities.
Since electricity and hydrogen prices are dual variables derived only after reaching equilibrium, providing an a priori analytical proof of IR is infeasible. Alternatively, this property is empirically verified by 
perturbation simulations; see Section \ref{sec:incentive}.

\begin{figure}[tb]
	\centering
	\includegraphics[width=4in]{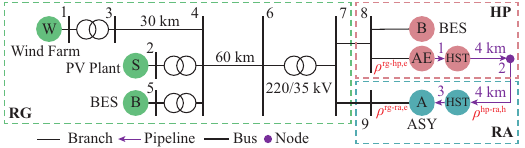}
	\caption{Topology of the ReP2A system used in the case study}
	\label{fig:case}
\end{figure}

\begin{figure}[t]
	\centering
	\includegraphics[width=3.8in]{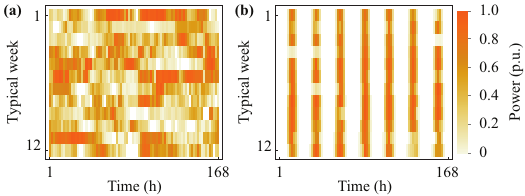}
	\caption{Typical weeks of RES in the case study. (a) Wind and (b) solar power.}
	\label{fig:WS}\vspace{-0pt}
\end{figure}

\section{Case studies}
\label{sec:cases}

In this section, the proposed incentive mechanism is evaluated through case studies. All the simulations are conducted in \emph{Wolfram Mathematica 14.0} on a laptop with an \emph{Intel Core Ultra 7 165H@3.80 GHz} CPU and 32 GB of RAM. In the hierarchical game, the inner- and outer-level problems are solved using \emph{Mosek 11.0} and \emph{Gurobi 12.02}, respectively.

\begin{table}[tb]\footnotesize
	\renewcommand{\arraystretch}{1.25}
	\caption{Equipment capacity parameters in case studies}\vspace{6pt}
	\label{tab:capacity}
	\centering
	\begin{tabular}{cc|cc}
		\hline\hline
		Parameter                          & Value                 & Parameter                   &Value            \\
		\hline
		$W^{\text{rg,wt}}/W^{\text{rg,pv}}$         &300/100 MW            &$W^{\text{rg,bes}}/W^{\text{hp,bes}}$      & 150/50 MWh  \\          $W^{\text{hp,ae}}$         & 150 MW            &$W^{\text{hp,hst}}/W^{\text{ra,hst}}$       & 1/2$\times10^5$  Nm$^3$       \\ $W^{\text{ra,asy}}/W^{\text{ga,asy}}$         & 15.66/78.3 t/h      &$W^{\text{ra,ast}}$         & 1000 t           \\
		\hline\hline
	\end{tabular}\vspace{10pt}
	
	\renewcommand{\arraystretch}{1.25}
	\caption{Carbon market mechanism benchmarks for comparisons}\vspace{6pt}
	\label{tab:case}
	\centering
	\begin{tabular}{cccc}
		\hline\hline
		\textbf{Mechanism}         & \tabincell{c}{CA cap for\\gray ammonia}                 & \tabincell{c}{CA incentive for\\green ammonia}                   &\tabincell{c}{CA\\transactions}          \\
		\hline
		M1              & \ding{55}      &\ding{55}     & \ding{55}                 \\
		M2              & \ding{51}      &\ding{55}     & \ding{55}                 \\
		M3              & \ding{51}      &\ding{51}     & Fixed price               \\ \hline
		\textbf{PCIM}   &\ding{51}       &\ding{51}     & \textbf{Equilibrium price}  \\
		\hline\hline
	\end{tabular}
\end{table}

\subsection{Case setups}
\label{sec:setups}

The case studies use an off-grid ReP2A system and 12 representative weeks of wind and solar data from a real-world project in northern China \cite{zeng2025planning}, as shown in Figs.~\ref{fig:case} and \ref{fig:WS}.
The capacity parameters are listed in Table~\ref{tab:capacity}, and all the operational parameters follow those of our previous work \cite{zeng2025planning}.

In the ammonia market, the maximum ammonia price $\rho^\text{max}$ is set to 2,900~CNY/t, and the price-elasticity factor $k^\text{am}$ is 35 t$^2$/CNY. The carbon intensity and production cost of GA are 3~t CO$_2$/t NH$_3$ \cite{yu2023china,shin2025comparative} and 2,000~CNY/t \cite{shin2025comparative}, respectively. The annual total CA is determined via the grandfathering method \cite{zhang2024grandfather}, using the given emission intensity, a benchmark yield (90\% of capacity), and a 3\% annual reduction factor. This approach yields $3\times78.3\times168\times12\times0.9\times0.97\approx413$ kt. Then, historical production data \cite{2024chinacarbon, zhao2010long} determine the initial and incentive CA, allocating $q^{\text{allo}}=344$ kt to GA and $q^{\text{rewa}}=69$ kt to ReP2A.

\begin{table*}[tb]
	\renewcommand{\arraystretch}{1.3}
	\footnotesize
	\caption{Operational performance comparison of different carbon transaction mechanisms}\vspace{6pt}
	\label{tab:comparsion1}
	\centering
	\begin{tabular}{cccccccc}
		\hline\hline
		\textbf{Mechanism}    & \tabincell{c}{ReP2A revenue\\(10$^7$ CNY)} & \tabincell{c}{GA revenue\\(10$^7$ CNY)}         &\tabincell{c}{CA traded\\(10$^3$ t)}     &\tabincell{c}{CA price\\(CNY/t)}   &\tabincell{c}{Gray/green\\ammonia\\yield (10$^3$ t)}  &\tabincell{c}{Average\\ammonia\\ price (CNY/t)}  & \tabincell{c}{Carbon\\emissions\\(10$^3$ t)} \\
		\hline
		M1        & 4.40 (baseline)       &7.59         & 0      & 0      & 157.9/18.5  &2481.0  &474     \\
		M2        & 4.59 \textbf{(+4.3\%)}       &6.70         & 0      & 103.5  & 114.8/18.5  &2583.2  &344    \\
		M3        & 4.66/4.83/5.04  &7.11/6.94/6.73    &69      & \tabincell{c}{25/50/80\\(\textbf{Fixed})}  & 137.8/18.5  &2528.6  &413 \\ \hline
		\textbf{PCIM}        &4.95 \textbf{(+12.5\%)}       &6.82         &69      & 67.1  & 137.8/18.5  &2528.6  &413\\
		\hline\hline
	\end{tabular}
\end{table*}

\subsection{Outer-level carbon and ammonia market equilibrium}
\label{sec:outerequilibrium}

\subsubsection{Effectiveness of the carbon incentive mechanism}

Four market mechanisms (M1--M3 and PCIM, detailed in Table~\ref{tab:case}) are compared. The operational outcomes are listed in Table \ref{tab:comparsion1}.

Under M1, which represents the current market without carbon constraints,  GA operates at full capacity, leading to a low ammonia price and limited ReP2A revenue.
Under M2, the carbon cap restricts GA output, increasing both ammonia prices and green ammonia revenue. However, ReP2A revenue improves only modestly by 4.3\%, and the GA utilization rate decreases to 72.7\%, below the typical 90\% \cite{yu2023optimal,yu2024optimal}, which is less desirable both technically and economically.

Although the ammonia revenue of the ReP2A  decreases, carbon trading revenues more than compensate, increasing total ReP2A revenue by 12.5\% and maintaining GA utilization at 87.3\%.
Compared with M1, while the total ammonia revenue of GA and ReP2A decreases from $11.99$ to $11.77 \times 10^7$ CNY (-1.8\%) under PCIM, carbon emissions decrease by 12.9\%, which is an acceptable tradeoff from a societal perspective.

Subsequently, M3 fixes the carbon price.
The results reveal that only ReP2A revenue varies with the carbon price; all the other outcomes remain unchanged.
Fig.~\ref{fig:Carbonpr} shows how the revenues of ReP2A and GA 
change under different carbon prices. We can see that the total ammonia revenue always reaches an optimum despite price changes, implying that regulators can reallocate profit between GA and ReP2A by selecting carbon prices without reducing sectorwide welfare.
If the price is less than 15 CNY/t or greater than 84 CNY/t, one party (either ReP2A or GA stakeholders) loses participation incentives. Under free carbon trading (the case of PCIM), the equilibrium price is 67.1 CNY/t, which lies within the mutually beneficial range [15, 84] CNY/t.

\begin{figure}[t]
	\centering
	\includegraphics[width=3.8in]{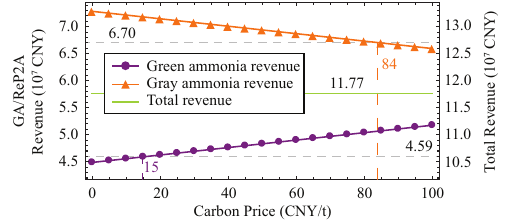}
	\caption{The revenues of green and gray ammonia and the sectorwide total revenue under different fixed carbon prices (case M3).}
	\label{fig:Carbonpr}
\end{figure}

\subsubsection{Ammonia market equilibrium analysis}

\begin{figure}[t]
	\centering
	\includegraphics[width=4.1in]{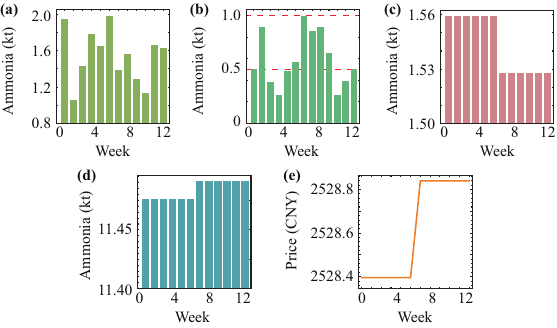}
	\caption{Equilibrium ammonia trading. (a) Green ammonia yield. (b) Green ammonia storage. (c) Green ammonia sales. (d) GA sales. (e) Ammonia price.}
	\label{fig:outer}
\end{figure}

The equilibrium outcomes under the PCIM are shown in Fig.~\ref{fig:outer}. As described in Section~\ref{sec:method}, weekly green ammonia output follows RES availability, and the AST smooths fluctuations to maintain stable sales. GA production operates at its CA-constrained maximum due to its lower cost.
With fixed GA output, the weekly sales of GA and green ammonia adjust to maximize profit under linear price elasticity. When AST cannot fully balance green ammonia sales, GA backfills demand.

Note that despite the positive effect of the PCIM on green ammonia production, incentive misalignment may arise among ReP2A stakeholders, necessitating inner-level analysis and the PCAM for ensuring IR, as discussed in Section \ref{sec:innerequilibrium}.

\subsection{Inner-level electricity and hydrogen trading equilibrium}
\label{sec:innerequilibrium}

\subsubsection{Equilibrium analysis}

\begin{figure}[t]
	\centering
	\includegraphics[width=3.8in]{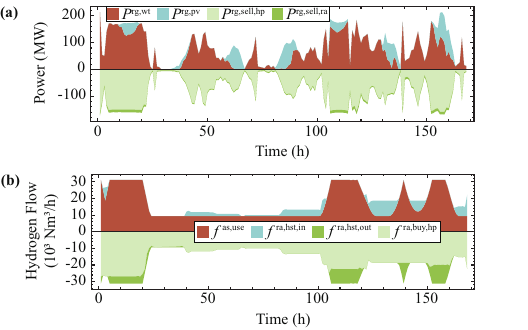}
	\caption{Equilibrium operation of RG/HP/RA in the 7th week under the PCIM. (a) Power generation and load. (b) Hydrogen balance in the ammonia plant}
	\label{fig:inner}
\end{figure}

\begin{figure}[t]
	\centering
	\includegraphics[width=3.8in]{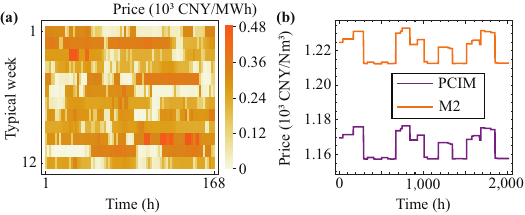}
	\caption{(a) Equilibrium electricity price under the PCIM. (b) Equilibrium hydrogen price under the M2 and PCIM}
	\label{fig:price}
\end{figure}

Fig.~\ref{fig:inner} presents the equilibrium operation of RG, HP, and RA during the 7th week under the PCIM.
Because the ammonia price does not influence weekly ammonia production (as explained in Section~\ref{sec:method}), M2 yields identical results. The trading and operation strategies of RG, HP, and AS follow RES profiles, bridging the gap between the volatility of RES and stable chemical production.

Electricity and hydrogen prices under different incentive mechanisms are shown in Fig.~\ref{fig:price}.
Combining the data in Figs. \ref{fig:price}(a) and \ref{fig:WS}, it is clear that electricity prices vary inversely with RES output.
In Fig.~\ref{fig:price}(b), although ammonia trading remains unchanged, hydrogen prices move with respect to ammonia prices (which is similar to electricity prices). Compared with M2, the PCIM reduces the ammonia price slightly (to $\frac{2528.6}{2583.2}\approx0.98$ times) but decreases the hydrogen price more notably (to $\frac{1.165}{1.22}\approx 0.95$ times), indicating that carbon trading redistributes benefits.

\subsubsection{Necessity of the CA allocation mechanism}

To demonstrate the importance of individually rational CA allocation among the stakeholders in the ReP2A chain, we compare the following CAMs under the PCIM (relative to M2) as follows:
\begin{itemize}
	\item \textbf{PCAM}: the proposed individually rational CAM; see Section \ref{sec:CAAM}.
	\item \textbf{CAM1}: all CA revenues allocated to a single stakeholder (e.g., RA).
	\item \textbf{CAM2}: CA revenues evenly allocated among stakeholders.
\end{itemize}

The revenues of each stakeholder under different CAMs are summarized in Table \ref{tab:caam}.
Carbon trading lowers ammonia prices and electricity and hydrogen LMPs, reducing stakeholders' revenues.
Under the CAM1, concentrating carbon revenue (0.46 $\times$ 10$^7$ CNY) in RA results in losses for RG and HP stakeholders, eliminating their incentive to participate. The same applies when the revenue is allocated solely to others.
CAM2 benefits all parties but disproportionately favors HP and RA; the gain by RG (+0.4\%) is too small relative to the added operational complexity of carbon trading and certification.
In contrast, the PCAM yields balanced improvements for all stakeholders, each receiving at least +5.2\%, thus maintaining high-level willingness to participation.

\begin{table}[t]
	\renewcommand{\arraystretch}{1.3}
	\footnotesize
	\caption{Revenues of each stakeholder under different CAMs}\vspace{6pt}
	\label{tab:caam}
	\centering
	\begin{threeparttable}
		\begin{tabular}{cccccc}
			\hline\hline
			\textbf{Mechanism}      & \textbf{CAM}  & \tabincell{c}{RG revenue (10$^7$ CNY)}      &\tabincell{c}{HP revenue (10$^7$ CNY)}  &\tabincell{c}{RA revenue (10$^7$ CNY)} &\tabincell{c}{Willingness}        \\
			\hline
			M2          &  /                &2.67        & 1.81    &0.11 \footnote      & \\
			\multirow{3}{*}{PCIM} &PCAM         &2.81 (+5.2\%)        &1.91 (+5.5\%)     &0.23 (+109.1\%)       &\tabincell{c}{\ding{51} (high)}      \\
			& CAM1                  &2.53 (--)\footnote        &\tabincell{c}{1.73 (--)}    &\tabincell{c}{0.69 (+)\\{\color{gray}(0.23+0.46)}}  &\ding{55}  \\
			& CAM2                  &\tabincell{c}{2.68 (+0.4\%)\\{\color{gray}(2.53+0.15)}}        &\tabincell{c}{1.88 (+3.9\%)\\{\color{gray}(1.73+0.15)}}     &\tabincell{c}{0.38 (+245.5\%)\\{\color{gray}(0.23+0.15)}}  &\ding{51} (low)               \\
			\hline\hline
		\end{tabular}
		\begin{tablenotes}
			\footnotesize
			\item[1] Due to heterogeneous flexibility, the revenue of the RA is relatively low. This was addressed in our work \cite{zeng2025planning},  so here we focus on only the effectiveness of the PCAM.
			\item[2] The revenues of RG, HP, and RA stakeholders without carbon revenue in PCIM are 2.53, 1.73, and 0.23$\times10^7$ CNY, respectively.
		\end{tablenotes}
	\end{threeparttable}
\end{table}

\begin{table}[t]
	\renewcommand{\arraystretch}{1.3}
	\footnotesize
	\caption{Variations in total ReP2A revenue and stakeholder gains (10$^7$ CNY) under perturbed CA trading volumes}\vspace{6pt}
	\label{tab:incentive}
	\centering
	\begin{threeparttable}
		\begin{tabular}{ccccc}
			\hline\hline
			CA trading (10$^3$ t)      & \tabincell{c}{ReP2A revenue}  & \tabincell{c}{RG revenue}      &\tabincell{c}{HP revenue}  &\tabincell{c}{RA revenue}\\
			\hline
			69 (equlibirum)          &  4.951                &2.816        & 1.909    &0.226        \\
			59         &  4.929                &2.813          &1.907     &0.210        \\
			49        &  4.897                &2.804         &1.901    &0.193     \\
			39		    & 4.855                  &2.788        &1.890     &0.177                 \\
			29		    & 4.802                  &2.766        &1.875     &0.161                 \\
			19		    & 4.738                  &2.738        &1.856     &0.145                 \\
			9		    & 4.664                  &2.703        &1.833     &0.129                 \\
			\hline\hline
		\end{tabular}
	\end{threeparttable}
\end{table}

\subsubsection{Individual Rationality of the PCAM}
\label{sec:incentive}

To test IR, a perturbation analysis is performed. Each participant is assumed to deviate by withholding CA sales, thereby reducing total CA trading. As shown in Table \ref{tab:incentive}, the total ReP2A revenue and all stakeholder revenues decrease monotonically as CA trading volume decreases. Consequently, no participant benefits from deviating, confirming that individual incentives are aligned with the system optimum under the PCAM.

\subsection{Impact of CA cap, green ammonia production capacity and number of GA producers}
\label{sec:sensitivity}

\begin{figure}[t]
	\centering
	\includegraphics[width=3.8in]{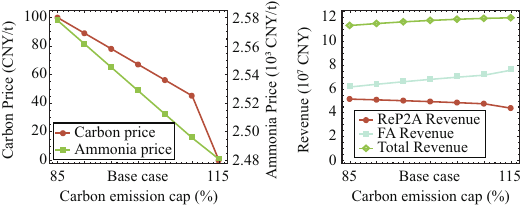}
	\caption{Impact of the carbon-emission cap on ReP2A, GA, total revenue, ammonia and carbon price}
	\label{fig:sensi}\vspace{10pt}
    \includegraphics[width=3.8in]{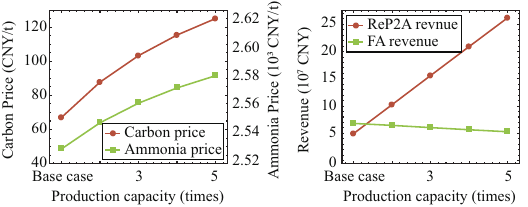}
	\caption{Price and revenue outcomes when green ammonia production capacity increases from 1 to 5 times the base case}
	\label{fig:sensi2}
\end{figure}

\subsubsection{Initial CA cap}

How different carbon caps shape stakeholder revenues and market outcomes is assessed in Fig.~\ref{fig:sensi}.
A tighter cap lowers total revenue but increases environmental performance, offering a reference for regulators seeking to balance the social value of ammonia against carbon-related costs.
As GA output decreases, both ammonia and carbon prices increase, increasing ReP2A revenue.
Because the CA cap affects total ammonia output, the ammonia price, and the carbon price through nearly linear relationships, the cap and the carbon price remain largely linearly correlated.
When the cap becomes too loose and GA production hits its capacity limit, the CAs of ReP2A lose value, and the carbon price falls to zero.
These insights provide practical guidance for setting and allocating CAs.

\subsubsection{Capacity of green ammonia production}

With the ongoing expansion of the ammonia industry, new capacity is expected to shift toward green ammonia. Equilibrium outcomes as green ammonia capacity increases from one to five times the base case are examined in Fig.~\ref{fig:sensi2}.
Because ReP2A facilities operate at lower utilization levels than does GA, a larger green ammonia share increases both the ammonia price and the carbon price.
Fig.~\ref{fig:sensi2} also shows declining GA revenue, reflecting a higher effective carbon cost than in that of the base case.
Under the proposed CA allocation and game framework, green ammonia therefore gains a stronger competitive position as its capacity expands and carbon prices continue to increase.

\begin{figure}[t]
	\centering
	\includegraphics[width=3.8in]{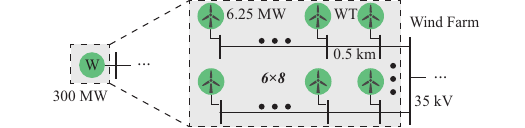}
	\caption{The detailed topology of the wind farm in a real-world ReP2A project}
	\label{fig:largecase}
\end{figure}

\subsubsection{Number of GA producers}

To examine the impact of multiple GA producers in the regional ammonia market, three producers, denoted as GA1, GA2, and GA3, are assumed, each with a production capacity of 26.1 t/h, with all other parameters unchanged. The operational results with three GA producers participating in CA transactions are shown in Table \ref{tab:3GA}. It can be observed that a prisoner's dilemma arises, where if any GA producer chooses not to participate in CA transactions, its revenue decreases, while if none of the three participates in CA transactions, their revenues are jointly maximized. However, when the three act jointly (i.e., PCIM in Table \ref{tab:comparsion1}) and participate in CA transactions, the average revenues becomes higher ($\frac{6.82}{3}\approx2.27 \times 10^7$ CNY). Therefore, GA producers are encouraged to form internal coalitions, which can enhance their own payoffs while promoting carbon trading.

\begin{table}[tb]
	\renewcommand{\arraystretch}{1.5}
	\footnotesize
	\caption{Operational results with three GA producers participating in CA transactions}\vspace{6pt}
	\label{tab:3GA}
	\centering
	\begin{tabular}{cccccc}
		\hline\hline
		\tabincell{c}{Participants in\\CA transactions}   & \tabincell{c}{ReP2A revenue\\(10$^7$ CNY)} & \tabincell{c}{GA1/GA2/GA3\\revenue (10$^7$ CNY)}         &\tabincell{c}{CA price\\(CNY/t)}   &\tabincell{c}{Gray/green ammonia   \\yield (10$^3$ t)}  &\tabincell{c}{Average ammonia\\ price (CNY/t)}  \\
		\hline
		No GA producers   & 4.59       &2.23/2.23/2.23         & 0                &114.8/18.5   &2583.2       \\
		GA2, GA3          & 5.43       &2.02/2.16/2.16         & 136.8             & 137.8/18.5  &2528.6      \\
		GA3               & 4.53       &2.10/2.10/2.89         &0                  & 129.2/18.5  &2549.1   \\
		\hline
		\tabincell{c}{GA1, GA2, GA3}        &5.45        &2.11/2.11/2.11    &139.8       & 137.8/18.5  &2528.6  \\
		\hline\hline
	\end{tabular}\vspace{4pt}
\end{table}

\subsection{Computational performance}
\label{sec:computation}

\subsubsection{Scalability}

The solution process includes 12 SPs of ammonia production, the outer equilibrium problem (\ref{eq:probfinal}), and the inner equilibrium problem (\ref{eq:problem}) (temporally decoupled into 12 parts). The solution time for each stage is shown in Table \ref{tab:time}. Since the SPs and (\ref{eq:problem}) are both LPs, and (\ref{eq:probfinal}) is a small-scale MILP, all can be solved efficiently and are scalable. Then, the detailed topology of the wind farm (as shown in Fig. \ref{fig:largecase}) is incorporated to extend the 9-bus system to a 57-bus system (real-world ReP2A projects do not feature complex topologies), thereby increasing the computational complexity for testing. The total computation time is about 913.1 s, validating the scalability of the proposed solution approach. Moreover, when multiple ReP2A and GA producers are considered, the inner problems are not coupled and can be solved in parallel, while the outer problem remains computationally tractable, ensuring overall computational efficiency.

\begin{table}[t]
	\renewcommand{\arraystretch}{1.3}\footnotesize
	\caption{Solution Time (s) of the hierarchical game model}\vspace{6pt}
	\label{tab:time}
	\centering
	\begin{tabular}{ccccc}
		\hline\hline
		Scale         & \tabincell{c}{12 ammonia\\production SPs}     & \tabincell{c}{Outer\\equilibrium}    &\tabincell{c}{Inner\\equilibrium}  &Total        \\
		\hline
		9-bus          &7$\times$12       &1.1      & 10$\times$12    &205.1 \\
		57-bus         &33$\times$12      &1.1      & 43$\times$12    &913.1      \\
		\hline\hline
	\end{tabular}
\end{table}

\begin{table}[t]
	\renewcommand{\arraystretch}{1.5}\footnotesize
	\caption{Revenues of ReP2A and GA under different objectives in the equilibrium solution problem (Eq. (\ref{eq:probfinal}))}\vspace{6pt}
	\label{tab:stability}
	\centering
	\begin{tabular}{cccc}
		\hline\hline
		Objective(10$^7$ CNY)         & 1                 & \tabincell{c}{Maximizing green ammonia\\and gray ammonia revenues}    &\tabincell{c}{Minimizing green ammonia\\and ray ammonia revenues}          \\
		\hline
		\tabincell{c}{Green ammonia revenue}          &4.95       &4.95      & 4.95  \\
		\tabincell{c}{Gray ammonia revenue}         & 6.82         &6.82       & 6.82       \\
		\hline\hline
	\end{tabular}
\end{table}

\subsubsection{Solution stability}
\label{sec:solutionstability}

Since the equivalent optimization problem corresponding to the inner Nash equilibrium has a unique solution \cite{rosen1965existence}, the outer Nash-Cournot equilibrium is also unique \cite{metzler2003nash}, making the solution of problem (\ref{eq:probfinal}) theoretically unique. To verify this, we obtain the equilibrium by using different objective functions, and the results are given in Table \ref{tab:stability}. It can be seen that the proposed hierarchical game yields a stable equilibrium and provides valuable insights.

\section{Conclusions}
\label{sec:conclusions}

In this work, carbon transactions are incorporated into the competition between green and gray ammonia producers and the interactions among ReP2A stakeholders are modeled through a hierarchical game. The carbon market design improves the market position of green ammonia, while the CA allocation mechanism maintains individual rationality within the ReP2A system. The main conclusions and policy implications are listed as follows:

\begin{enumerate}
	\item A tighter CA cap lowers total ammonia revenue but increases environmental performance to a much greater extent, offering regulators a benchmark for weighing the social value of ammonia against its carbon cost. Under a fixed CA quota, total revenue is maximized under the PCIM.
	
	\item Because carbon pricing does not change the equilibrium of other markets, regulators may adjust the carbon price to shift profits between the ReP2A and GA systems. However, higher carbon prices do not always benefit green ammonia and may suppress transactions. The feasible carbon price interval identified here provides actionable guidance for policy design.
	
	\item The greater overall profitability of ReP2A does not guarantee greater returns for each entity along the process chain. Carbon trading reshapes ammonia prices, which then influence the equilibrium of electricity and hydrogen markets.
	
	\item Effective carbon-revenue allocation is critical for sustaining incentives. Equal or centralized allocations fail to ensure universal benefits, whereas the PCAM supports balanced incentives and stronger participation.
\end{enumerate}

Future work will investigate the dynamic transition from gray to low-carbon ammonia and design multistage, multi-type market and policy frameworks that accelerate long-term decarbonization in the energy, chemical, and maritime sectors.

\section*{Acknowledgement}

The authors gratefully acknowledge the financial support came from the National Natural Science Foundation of China (52377116, 72103167 and 52577129).

\section*{Declaration of Interest}
None.

\section*{Data Availability}

The data related to this work are available upon request.



\end{document}